\documentclass[11pt]{amsart}
\usepackage{amscd,amsmath,amssymb,amsfonts,verbatim}
\usepackage[cmtip, all]{xy}
\usepackage{comment}
\usepackage{tikz-cd}

\newtheorem{thm}{Theorem}[section]

\newtheorem{lem}[thm]{Lemma}
\newtheorem{cor}[thm]{Corollary}

\theoremstyle{definition}

\newtheorem{defn}[thm]{Definition}
\theoremstyle{remark}
\newtheorem{remk}[thm]{Remark}
\newtheorem{remks}[thm]{Remarks}

\newtheorem{exm}[thm]{Example}
\newtheorem{exms}[thm]{Examples}
\newtheorem{notat}[thm]{Notation}
\numberwithin{equation}{section}

{\hfill$\square$\end{defn}}
{\hfill$\square$\end{remk}}
{\hfill$\square$\end{remks}}
{\hfill$\square$\end{exm}}
{\hfill$\square$\end{exms}}
{\hfill$\square$\end{notat}}

\newcommand{\remove}[1]{}

\newcommand{\rank}{{\rm rank}}

\newcommand{\Hom}{{\rm Hom}}

\newcommand{\Spec}{{\rm Spec \,}}

\newcommand{\Sup}{{\rm Sup \ }}

\newcommand{\Ass}{{\rm Ass}\,}

\newcommand{\hh}{\rm ht}
\newcommand{\Um}{\mbox{\rm Um\,}}

\newcommand{\ds}{{/\kern-3pt/}}

\renewcommand{\dim}{\text{\rm dim}}

\newcommand{\tuborg}{\left\{\begin{array}{ll}}
\newcommand{\sluttuborg}{\end{array}\right.}

\newcounter{elno}

\newcounter{elno-abc}   

\newcounter{elno-abc-prime}

\begin{document}

\title[Existence of unimodular element in a projective module]{Existence of unimodular element in a projective module over symbolic Rees algebras}
\author{Chandan Bhaumik and Husney Parvez Sarwar}


\address{(Chandan Bhaumik) Department  of Mathematics, Indian  Institute of Technology Kharagpur,  Kharagpur 721302, West Bengal, India}
\email{cbhaumik11math@gmail.com}

\address{(H.P. Sarwar) Department  of Mathematics, Indian  Institute of Technology Kharagpur,  Kharagpur 721302, West Bengal, India}
\email{parvez@maths.iitkgp.ac.in}
\email{mathparvez@gmail.com}




\maketitle

\begin{quote}\emph{Abstract.}  
 Let $A$ be a symbolic (or an extended symbolic) Rees algebra (need not be Noetherian) of dimension $d$. Let $P$ be a  finitely generated projective $A$-module of rank $\geq$  $d$. Then P has a unimodular element. This improves the classical result of Serre for the mentioned class of algebras.
\end{quote}
\setcounter{tocdepth}{1}

\section{Introduction}\label{sec:Intro}

Let $R$ be a commutative ring and $P$ a finitely generated projective $R$-module.
An element $p\in P$ is said to be a unimodular if there exists $\phi \in \Hom(P,R)$ such that $\phi(p)=1$, in other words, $P$
splits off a free summand of rank one, i.e. $P\cong Q \oplus R$ 
for some projective $R$-module $Q$. If $R$ is a Noetherian commutative ring of dimension $d$, then
a classical result of Serre \cite{Serre58} asserts that every finitely generated projective $R$-modules of rank $>d$ has a unimodular element. In fact, this is the best possible result in general as a common example which can be found easily in the literature is the tangent bundle over the real algebraic sphere. Therefore if the $\rank(P)\leq d$,  then the  question that $P$ has a unimodular element, is subtle. In the case $\rank(P)=\dim(R)$, where  $R$ is an affine algebra over an algebraically closed field, there is a well developed obstruction theory which is due to Murthy \cite[Theorem 3.7]{Mu94}. (For a proof of the hypotheses in \cite[Theorem 3.7]{Mu94}, see \cite[Corollary 1.5]{Kr19}).

Other than Murthy's works, there are various fascinating obstruction theory 
in the literature when $rank(P)\leq \dim(R)$.
For this, we recommend the reader to look at MathSciNet citations of Murthy's
paper \cite{Mu94}. We do not deal with such a fancy theory here as our paper is in the pursuit of discovering commutative rings where such an obstruction does not exist.

In this paper, we show that symbolic Ress algebra and extended symbolic Rees algebras (see Definition \ref{defn:sym}) are examples of commutative rings where such an obstruction does not exist. The importance of symbolic Rees algebra first comes from the counterexample of Hilbert's 14th problem. Rees \cite{Rees58} provided the first counterexample to Zariski's version of Hilbert's 14th problem by proving an ideal $J$ of $R$ whose symbolic Rees algebra $\mathcal{R}_{s}(J)$ is not finitely generated. 
We prove the following result about an existence of a unimodular element in a projective module over a symbolic Rees algebra or over an extended symbolic Rees algebra.
\begin{thm}(Theorem \ref{uesra})
Let $R$ be a commutative noetherian domain of dimension $d$ and $I$ an ideal of $R$. Let $A= \mathcal{R}_{s}(I)$ or $\mathcal{R}_{s}(I,x^{-1})$ (symbolic or extended symbolic Rees algebra)(see Definition \ref{defn:sym}) and $P$ a finitely generated projective $A$-module of rank $\geq d+1$. Then $P$ has a unimodular element, i.e.
 $P\cong Q\oplus A$ for some projective $A$-module $Q$.
  \label{ecsra}
\end{thm}
 The similar results are obtained for polynomial rings by Plumstead \cite{Plum83}, for Laurent polynomial rings by Mandal \cite{Man82} and for overring of polynomial rings by Rao \cite{Rao82}. For Rees algebras, the analogous
conjectures are considered in \cite{RS19}. The above results are generalized to polynomial and Laurent polynomial rings for several variables by Bhatwadekar--Roy \cite{BR84} and Bhatwadekar--Lindel--Rao \cite{BLR85}. For monoid extension, existence of unimodular problem is considered by
Sarwar \cite{Sarwar16} \cite{Sarwar21},
Keshari--Sarwar \cite{KeSa17}, and Keshari--Mathew \cite{KeMa22}.

Let $R$ be a commutative Noetherian domain of Krull dimension $d$. Then the Rees algebras over $R$ are Noetherian but symbolic Rees algebras need not be Noetherian (for example, see \cite{Rob85}). However we can prove the Theorem \ref{ecsra} without using the assumption that symbolic Rees algebras are Noetherian. This is possible because of Heitmann's result \cite{Heit84} over non-Noetherian rings. For an extension of Heitmann's result, see the results of Gupta \cite{Gup16}. Our method is to use standard sheaf patching techniques. More precisely, we find two covers, then we prove the results separately on each cover, finally we patch/glue to get the results over the original space. A similar technique is used in \cite{RS19}. However there to prove the result in one of the two covers, they have used Plumstead's generalized dimension function. Here we prove the similar result using Heitmann's result as symbolic Rees algebras need not be Noetherian. In section $2$, we have recalled some basic definitions and results which will be used throughout the paper, also we have studied the Krull dimension of (extended) symbolic Rees algebras via valuation dimension. Theorem \ref{uesra} is proved in section $3$.

\section{Recollection of basic definitions and results}\label{sec:Prelim}

\subsection{Notation.} Throughout the paper, we assume that the rings are commutative with the unity and the modules are finitely generated. Also, we assume that projective module have the constant rank function.\par Now we recall some basic notions. For a ring $R$, let $\dim R$ denote the Krull dimension of $R$. Let $\Ass(R)$ denote the set of associated prime ideals of $R$, and $\hh (I)$ denote the height of the ideal $I$. 

\begin{defn}[Unimodular Element]
Let $R$ be a ring $R$ and $M$ an $R$-module. For $m \in M$, we define an ideal $O_{M}(m)=\{\varphi(m) : \varphi \in M^{*}=\Hom_{R}(M,R)\}$ of $R$ which is called the order ideal. The element $m$ in $M$ is called a unimodular element if $O_{M}(m)=R$, i.e. there exists a surjective $R$-linear map $\varphi \in M^{*}$ such that $\varphi(m)=1$. 
Let $\Um(M)$ denote the set of all unimodular elements in $M$.
\end{defn}

\textbf{Valuation dimension.} The notion of the valuation dimension was introduced by Jaffard \cite{Jaf60}. Let $A$ be an integral domain, an overring $B$ of $A$ is a subring of the field of fractions $K$ of $A$ that contains $A$, i.e. $A\subseteq B \subseteq K$. \par 
Valuation overring of $A$ means an overring of $A$ which is a valuation ring.
\begin{defn}[Valuation Dimension]
Let $R$ be an integral domain, the valuation dimension of $R$ is the supremum of dimensions of the valuation overring of $R$. Valuation dimension of $R$ is denoted by $\dim_{v}R$, i.e. \[\dim_{v}R=\Sup\{\dim V : V \ \text{is a valuation overring of} \ R\}.\]
\end{defn}
\begin{thm}\cite[Theorem 0.1]{ABDFK88}
Let $R$ be an integral domain which is not a field, $K$ is a field of fraction of $R$. Let $F$ denote the algebraic extension of $K$, and $d$ a positive integer. Then we say, $R$ have finite valuation dimension $d$ (write $\dim_{v}R=d$) if the following equivalent conditions are satisfied: \begin{enumerate}
    \item Each valuation overring of $R$ in $F$ has dimension $\leq d$ and there exists a valuation overring of $R$ in $F$ of dimension $d$.
    \item Each overring of $R$ in $F$ has dimension $\leq d$ and there exists a overring of $R$ in $F$ of dimension $d$.
    \item $\dim R[x_{1},\ldots,x_{d}]=2d$.
    \item $\dim R[x_{1},\ldots,x_{n}]=n+d$ if $n \geq d-1$.
\end{enumerate}
We say that $\dim_{v}R=\infty$ if there exists no $d$ satisfying the conditions $1$ to $4$. For the sake of completeness, each field is assigned valuation dimension $0$.
\end{thm}
Clearly $\dim R\leq\dim_{v}R$ for every domain $R$ and if $B$ is an overring of $A$, then $\dim_{v}B\leq\dim_{v}A$.
\smallskip

\subsection{Recollection of Rees algebra and Symbolic Rees Algebras.}
 Let $R$ be a commutative ring and $I$ an ideal of $R$. The Rees algebra (also known as blow-up algebra) of $R$ with respect to $I$ as a subring of $R[x]$ define to be \[
    R[Ix] =\{\sum_{i=0}^{n} a_{i}x^{i} : a_{i} \in I^{i}\}= \bigoplus_{n\geq0} \ I^{n}x^{n} \subseteq R[x].
 \]
The extended Rees algebra of $R$ with respect to $I$ as a subring of $R[x,x^{-1}]$ define to be  \[
    R[Ix,x^{-1}] = \{\sum_{i=-n}^{n} a_{i}x^{i} : a_{i} \in I^{i}\} = \bigoplus_{n\in \mathbb{Z}} \ I^{n}x^{n} \subseteq R[x,x^{-1}].
 \] 
where for $n\leq 0$, $I^{n}=R$ by convention.\par
The following theorem is for the Krull dimension of Rees algebra and extended Rees algebra (see \cite[Theorem 5.1.4]{SH06}).
\begin{thm}
Let $R$ be a Noetherian ring and $I$ an ideal of $R$. If $\dim R$ is finite, then\begin{enumerate}
    \item $\dim R[Ix]$ = $\dim R+1$, if $I\nsubseteq \mathfrak{p}$ for some $\mathfrak{p}\in \Spec R$ with $\dim (R/\mathfrak{p}) = \dim R$ and $\dim R[Ix] = \dim R$, otherwise.
    \item $\dim R[Ix,x^{-1}] = \dim R+1$.
\end{enumerate}
\end{thm}
\begin{cor}
Let $R$ be a Noetherian domain of dimension $d$ and $(0)\neq I$ an ideal of $R$. Then $\dim_{v}R[Ix]=d+1$.\label{vdra}
\end{cor}
For further properties of Rees algebras $R[Ix]$ and extended Rees algebras $R[Ix,x^{-1}]$, one can see \cite{SH06}.
\begin{defn}\label{defn:sym}
[Symbolic Power, Symbolic Rees Algebra]
Let $R$ be a commutative Noetherian ring and $I$ an ideal of $R$. The $n$-th symbolic power of $I$ is an ideal of $R$ defined by \[I^{(n)}=\bigcap\limits_{\mathfrak{p}\in \Ass(R/I)}(I^{n}R_{\mathfrak{p}}\cap R).\]
\end{defn}
In general, given an ideal $I$ of $R$ the ordinary power $I^{n}$ do not coincide with the symbolic power $I^{(n)}$. In particular, if $I=\mathfrak{p}\in \Spec(R)$, then $\mathfrak{p}^{(n)}=\mathfrak{p}^{n}R_{\mathfrak{p}}\cap R$ and if $I$ is a maximal ideal $\mathfrak{m}$, then $\mathfrak{m}^{(n)}=\mathfrak{m}^{n}$. Clearly from the definition $I^{(1)}=I$, $I^{n}\subseteq I^{(n)}$ and $I^{(n+1)}\subseteq I^{(n)}$ for all $n\geq 1$.\par
Let $R$ be a commutative noetherian ring and $I$ an ideal of $R$. The symbolic Rees algebra of $R$ with respect to $I$ is a graded algebra defined as
\[
    \mathcal{R}_{s}(I)= \{\sum_{i=0}^{n} a_{i}x^{i} : a_{i} \in I^{(i)}\}= \bigoplus_{n\geq0} \ I^{(n)}x^{n} \subseteq R[x].
\]
Sometimes it is also known as symbolic blow-up algebra. In particular, for all $n\geq 2$ if $I^{(n)}=I^{n}$, then the symbolic Rees algebra $\mathcal{R}_{s}(I)$ coincide with the Rees algebra $R[Ix]$. \par

One can define the extended symbolic Rees algebra of an ideal $I$ of $R$ as an extension of $\mathcal{R}_{s}(I)$. Let us denote $\mathcal{R}_{s}(I,x^{-1})$ as extended symbolic Rees algebra define to be \[\mathcal{R}_{s}(I,x^{-1})= \{\sum_{i=-n}^{n} a_{i}x^{i} : a_{i} \in I^{(i)}\} = \bigoplus_{n\in \mathbb{Z}} \ I^{(n)}x^{n} \subseteq R[x,x^{-1}].\]
where for $n\leq 0$, $I^{(n)}=R$ by convention. \par
For further properties of $\mathcal{R}_{s}(I)$ and $\mathcal{R}_{s}(I,x^{-1})$, one can see \cite{BV88}, \cite{SH06}.
\begin{lem}
Let $R$ be a commutative Noetherian domain of dimension $d$. Then \[\dim \mathcal{R}_{s}(I)=\begin{cases}
d, & \text{if} \ I=(0); \\
d+1, & \text{otherwise}.
\end{cases}\]\label{dimsra}
\end{lem}
\begin{proof}
For $I=(0)$, there is nothing to prove. So we can assume $I\neq (0)$. By Corollary \ref{vdra}, $\dim_{v}R[Ix]=1+d=\dim R[Ix]$.
 Consider $R[Ix]\subseteq \mathcal{R}_{s}(I)\subseteq R[x]$. Then $\dim_{v} \mathcal{R}_{s}(I)\leq \dim_{v}R[Ix]=d+1 $ and $\dim_{v} \mathcal{R}_{s}(I)\geq \dim_{v}R[x]=d+1$. This implies $\dim_{v} \mathcal{R}_{s}(I)=d+1$. We know that $\dim \mathcal{R}_{s}(I)\leq\dim_{v} \mathcal{R}_{s}(I)$. Hence $\dim \mathcal{R}_{s}(I)\leq d+1$. \par 
For the other inequality, note that  $\mathcal{R}_{s}(I)=\mathcal{R}_{s}(I)_{0}\oplus\mathcal{R}_{s}(I)_{+} $, where $\mathcal{R}_{s}(I)_{0}=R$ and $\mathcal{R}_{s}(I)_{+}= Ix \oplus I^{(2)}x^{2}\oplus \cdots $, and $\mathcal{R}_{s}(I)/\mathcal{R}_{s}(I)_{+}\cong R$. Hence $\mathcal{R}_{s}(I)_{+}$ is a prime ideal of $\mathcal{R}_{s}(I)$ and $\hh(\mathcal{R}_{s}(I)_{+})>0$. Then we have  \[\dim(\mathcal{R}_{s}(I)/\mathcal{R}_{s}(I)_{+})+\hh(\mathcal{R}_{s}(I)_{+})\leq \dim\mathcal{R}_{s}(I).\] This implies that $\dim R+1\leq \dim\mathcal{R}_{s}(I)$. Therefore $\dim\mathcal{R}_{s}(I)=d+1$.
\end{proof}
\begin{lem}
Let $R$ be a Noetherian domain of dimension $d$. Then $\dim \mathcal{R}_{s}(I,x^{-1})=d+1.$\label{dimesra}
\end{lem}
\begin{proof}
Note that $\dim_{v}R[x^{-1}]=d+1$. Consider $R[x^{-1}]\subseteq \mathcal{R}_{s}(I,x^{-1})\subseteq R[x,x^{-1}].$ By \cite[Lemma 1.15]{ABDFK88}, $\dim_{v}R[x^{-1}]=\dim_{v}R[x,x^{-1}]$. Hence $\dim_{v} \mathcal{R}_{s}(I,x^{-1})=1+d$. Then we have $\dim \mathcal{R}_{s}(I,x^{-1})\leq\dim_{v} \mathcal{R}_{s}(I,x^{-1})=d+1$.
 \par
 For the other inequality, from $R[x^{-1}]\subseteq \mathcal{R}_{s}(I,x^{-1})\subseteq R[x,x^{-1}]$, we observe that  $\mathcal{R}_{s}(I,x^{-1})_{x^{-1}}= R[x,x^{-1}]$. Then we have \[\dim \mathcal{R}_{s}(I,x^{-1})\geq \dim \mathcal{R}_{s}(I,x^{-1})_{x^{-1}}=\dim R[x,x^{-1}].\]
 Hence $\dim \mathcal{R}_{s}(I,x^{-1})\geq d+1$. Therefore $\dim \mathcal{R}_{s}(I,x^{-1})=d+1$.
\end{proof}
If $R$ is a Noetherian ring, then we know that the Rees algebra is a finitely generated $R$-algebra. Unlike Rees algebra, the symbolic Rees algebra is not necessarily finitely generated $R$-algebra. For example, see \cite{Rob85}.
\section{Existence of a unimodular element}
\begin{lem}\label{local-commute}
Let $R$ be a Noetherian ring, $I$ an ideal of $R$, and  $T$ be a multiplicatively closed subset of $R$. Then $T^{-1}\mathcal{R}_{s}(I)=T^{-1}\mathcal{R}_{s}(T^{-1}I)$, where $T^{-1}\mathcal{R}_{s}(T^{-1}I)$ means the  symbolic Rees algebra of $T^{-1}R$ with respect to $T^{-1}I$.
\end{lem}
\begin{proof}
By definition $\mathcal{R}_{s}(I)=R\oplus Ix\oplus I^{(2)}x^{2}\oplus \cdots$,
and since the localization commutes with direct sums, we have
\[\begin{split}
   T^{-1}\mathcal{R}_{s}(I) & = T^{-1}R\oplus T^{-1}(I)x\oplus T^{-1}(I^{(2)})x^{2}\oplus \cdots \\
    & = T^{-1}R\oplus T^{-1}(IR)x\oplus T^{-1}(I^{(2)}R)x^{2}\oplus \cdots.
\end{split}\]
Now for $n\geq 1$, 
\[\begin{split}
    T^{-1}(I^{(n)})&=T^{-1}(\bigcap\limits_{\Ass(R/I)}(I^{n}R_{\mathfrak{p}}\cap R))\\ &= \bigcap\limits_{\Ass(T^{-1}(R/I))}T^{-1}(I^{n}R_{\mathfrak{p}}\cap R)\\  & =\bigcap\limits_{\Ass(T^{-1}R/T^{-1}I))}((T^{-1}I)^{n}(T^{-1}R)_{\mathfrak{p}}\cap T^{-1}R)\\ & = (T^{-1}I)^{(n)}.
\end{split}\]

Since the localization commutes with direct sums, we also have
\[ T^{-1}\mathcal{R}_{s}(T^{-1}I)= T^{-1}R \oplus T^{-1}I(T^{-1}R)x\oplus (T^{-1}I(T^{-1}R))^{(2)}x^{2}\oplus \cdots.\]

But we already have $(T^{-1}I)^{(n)}=T^{-1}(I^{(n)})$, hence we conclude that $T^{-1}\mathcal{R}_{s}(I)=T^{-1}\mathcal{R}_{s}(T^{-1}I)$.
\end{proof}

The following lemma is a particular case of the previous one.  We use the following version later.
\begin{lem}
Let $R$ be a commutative Noetherian domain and $(0)\neq I$ an ideal of $R$. Then \begin{enumerate}
    \item If $T :=$ set of all non-zero-divisors of $R$, $T^{-1}\mathcal{R}_{s}(I)=T^{-1}R[x]$. 
    \item If $T=\{1,a,a^{2},\ldots\}$ with $a\in I$ is a non-zero-divisor of $R$, $T^{-1}\mathcal{R}_{s}(I)=R_{a}[x]$. 
\end{enumerate}\label{locsra}
\end{lem}
\begin{proof}
$(1)$ Since $R$ is a domain, $T=R^{*}=R \backslash\{0\}$. For $n\geq 1$, $I^{(n)}$ is an ideal of $R$ and $T\cap I^{(n)}\neq\emptyset$. This implies $T^{-1}(I^{(n)})=T^{-1}R$. Hence by 1st part of the proof of Lemma \ref{local-commute}, we get $T^{-1}\mathcal{R}_{s}(I)=T^{-1}R[x]$.\par 
$(2)$ By \cite[Lemma 3.1]{RS19}, we have $T^{-1}R[Ix]=T^{-1}R[(T^{-1}I)x]$. Since $T\cap I\neq\emptyset$ this implies $T^{-1}R[Ix]=T^{-1}R[x]$. Consider \[R[Ix]\subseteq \mathcal{R}_{s}(I)\subseteq R[x].\]
After localize at $T$, we have \[T^{-1}R[Ix]\subseteq T^{-1}\mathcal{R}_{s}(I)\subseteq T^{-1}R[x].\] 
 Therefore $T^{-1}\mathcal{R}_{s}(I)=R_{a}[x]$.
\end{proof}

\begin{lem}\label{3.3}
Let $R$ be a commutative domain of  dimension $d$ and $A=R_{1+aR}$, where $a$ is a non-zero non-unit element of $R$. Let $P$ be a projective $R$-module of rank $\geq d$. Let $Q:=P_{1+aR}$. Then $Q$ has a unimodular element. 
\end{lem}

\begin{proof}
 Let `overbar' denote the going modulo the ideal $aR$. Since $(0)$ is the minimum prime ideal in the domain $R$, we observe that $dim(\overline{R})<d$. By \cite[Corollary 2.6]{Heit84},
$\overline{P}$ has a unimodular element. Therefore there exists $\overline{p}$ such that the order ideal  $O_{\overline{P}}(\overline{p})=\overline{R}$.
Let $p\in P$ be a lift of $\overline{p}$. Now observe that $1+aR \subset $ of the order ideal $O_{P}(p)$. Hence $p_{1+aR}$ is a unimodular element of $Q$. This finishes the proof.
\end{proof}

The following theorem is a result of Ravi A. Rao \cite[Theorem 5.1(I)]{Rao82}.
\begin{thm}
Let $R$ be a commutative Noetherian ring of dimension $d$, and $S$ be a multiplicative closed set of non-zero-divisors of $R[x]$. Let $A$ be a ring lying between $R[x]$ and $S^{-1}R[x]$. Then for $n\geq d+2$, E$_{n}(A)$ acts transitively on Um$(A^{n})$.\label{5.1}
\end{thm}

\begin{thm}
Let $R$ be a commutative Noetherian domain of dimension $d$ and $I$ an ideal of $R$. Let $A= \mathcal{R}_{s}(I)$ or $\mathcal{R}_{s}(I,x^{-1})$ and $P$ a finitely generated projective $A$-module of $\rank \geq d+1$.
Then $P$ has a unimodular element. \label{uesra}
\end{thm}
\begin{proof}
First, we assume that $A=\mathcal{R}_{s}(I)$. Let $\rank(P)=m>d$. Since $A$ is a subring of an integral domain $R[x]$,  $A$ is an integral domain. If $I=(0)$, then $I^{(n)}=(0)$ this implies $A=R$.  In this case, the theorem follows from Serre \cite{Serre58}. If $I=(1)$, then $I^{(n)}=(1)$ this implies $A=R[x]$. In this case, the theorem follows from Plumstead \cite[Corollary 4]{Plum83}. So we can  assume $I\neq (0)$ and $I\neq (1)$. By Lemma \ref{dimsra}, $\dim \mathcal{R}_{s}(I)=d+1$. 


Let $T$ be the set of all non-zero-divisors of $R$. By Lemma \ref{locsra}(1), $T^{-1}A=T^{-1}R[x]$, where $T^{-1}R$ is a quotient field of $R$. Since $T^{-1}R[x]$ is a PID, $T^{-1}P$ is a free module over $T^{-1}A$. Since $P$ is finitely generated, there exists $t\in T$ such that $P_{t}$ is a free $\mathcal{R}_{s}(I)_{t}$-module.  Let $0\neq a$ be a non-unit element of $I$. Then $P_{at}$ is a free $\mathcal{R}_{s}(I)_{at}$-module. \par
 Denote $b=at$. By Lemma \ref{locsra}(2), we have $\mathcal{R}_{s}(I)_{b}\cong R_{b}[x]$. Now consider the following Cartesian square of rings \begin{center}
    \begin{tikzcd}
    \mathcal{R}_{s}(I) \ar[r,"i_{1}"] \ar[d,"i_{2}"] & \mathcal{R}_{s}(I)_{b}\cong R_{b}[x] \ar[d,"j_{1}"]\\ \mathcal{R}_{s}(I)_{1+b\mathcal{R}_{s}(I)} \ar[r,"j_{2}"] & R_{b}[x]_{1+b\mathcal{R}_{s}(I)}.
    \end{tikzcd}
\end{center}
 Since $P_{b}$ is a free $A_{b}$-module of rank $m$, $P_{b}\cong A_{b}^{m}$. Hence $P_{b}$ has a unimodular element say $p_{1}$. Since $P_{1+bA}$ is projective $A_{1+bA}$-module, by Lemma \ref{3.3}(1),  $P_{1+bA}$ has a unimodular element say $p_{2}$. Hence we have $P_{1+bA}\cong p_{2}A_{1+bA}\oplus Q$, with projective  $A_{1+bA}$-module $Q$. \par
Now consider $B=A_{b(1+bA)}=(1+bA)^{-1}\mathcal{R}_{s}(I)_{b}=(1+bR)^{-1}D^{-1}R_{b}[x]=D^{-1}R^{'}[x]$, where $D$ is a multiplicative closed subset of $R^{'}[x]$ and $R^{'}=R_{b(1+bR)}$ is of dimension $d-1$. Since $P_{b(1+bA)}$ is a free $B$-module of rank $m$, $P_{b(1+bA)}\cong B^{m}$. Let $\Bar{u}$ and $\Bar{v}$ be the image of $p_{1}$ and $p_{2}$ respectively in $P_{b(1+bA)}$. Now consider the following cartesian square of projective modules.
\begin{center}
    \begin{tikzcd}
    P \ar[r,"i_{1}"] \ar[d,"i_{2}"] & P_{b} \ar[d,"j_{1}"]\\ P_{1+bA} \ar[r,"j_{2}"] & P_{b(1+bA)}\cong B^{m}
    \end{tikzcd}
\end{center}
Note that the ring $B$ lying between $R^{'}[x]$ and $S^{-1}R^{'}[x]$, where $S$ is a set of all non-zero-divisors in $R^{'}[x]$. Then by Theorem $\ref{5.1}$, there exists $\sigma \in$ E$_{m}(B)$ such that $\sigma(\Bar{u})=\Bar{v}$. By \cite[Corollary 3.2]{Roy82} (see also \cite[Proposition 3.2]{BLR85}), we can split $\sigma$ as $(\sigma_{1})_{b}(\sigma_{2})_{1+bA}$, where $\sigma_{2}\in$ E$(P_{b})$ and $\sigma_{1}\in$ E$(P_{1+bA})$. Since $(\sigma_{1})_{b}(\sigma_{2})_{1+bA}(\Bar{u})=\Bar{v}$, for suitably changing $p_{1}$ and $p_{2}$, we can assume that $\Bar{u}=\Bar{v}$. 

Now consider the following fiber product diagram of projective modules.
\begin{center}
\begin{tikzcd}[row sep=scriptsize, column sep=scriptsize]
& P \arrow[dl,twoheadrightarrow, "\psi"] \arrow[rr] \arrow[dd] & & P_{b}\cong A_{b}^{m} \arrow[dl, twoheadrightarrow, "\psi_{1}"] \arrow[dd] \\
A \arrow[rr, crossing over] \arrow[dd] & & A_{b} \\
& P_{1+bA}\cong A_{1+bA}\oplus Q \arrow[dl,twoheadrightarrow,"\psi_{2}"] \arrow[rr] & & P_{b(1+bA)} \arrow[dl] \\
A_{1+bA} \arrow[rr] & & A_{b(1+bA)} \arrow[from=uu, crossing over]\\
\end{tikzcd}
\end{center}
Since $\psi_{1}$ and $\psi_{2}$ are  surjective homomorphisms, from the above diagram, we get a surjective homomorphism $\psi : P \twoheadrightarrow A$. Hence there is $p\in P$ such that $\psi(p)=1$. Therefore $P$ has a unimodular element. \par
When $A=\mathcal{R}_{s}(I,x^{-1})$, the proof is a verbatim copy of the above. However one needs to use \cite[Theorem 2.1]{Man82} in the second corner. 
A diligent reader can work out the details.
\end{proof}

\textbf{Acknowledgement:} H.P. Sarwar would like to thank S.E.R.B. Govt. of India for the grant SRG/2020/000272.

\end{document}